\documentclass[a4paper]{amsart}
\usepackage[foot]{amsaddr}

\usepackage[latin1]{inputenc}
\usepackage{amssymb,amsmath,amsthm}
\usepackage{amsfonts}
\usepackage{amsaddr}
\usepackage{enumitem}
\usepackage{url}
\usepackage{hyphenat}
\usepackage{hyperref}
\hypersetup{colorlinks=true,linkcolor=[rgb]{0.2,0.2,0.2}}
\usepackage[english]{babel}

\theoremstyle{plain}
\newtheorem{theorem}{Theorem}[section]
\newtheorem{proposition}[theorem]{Proposition}
\newtheorem{lemma}[theorem]{Lemma}
\newtheorem{corollary}[theorem]{Corollary}
\theoremstyle{definition}
\newtheorem{definition}[theorem]{Definition}
\newtheorem{example}[theorem]{Example}
\newtheorem{problem}[theorem]{Problem}
\newtheorem{remark}[theorem]{Remark}

\newcommand{\card}[1]{\ensuremath{\lvert{#1}\rvert}}
\newcommand{\gen}[1]{\ensuremath{\langle{#1}\rangle}}

\newcommand{\genHS}[1]{\ensuremath{\langle{#1}\rangle^\mathrm{HS}}}

\begin{document}
\title{HS-stability and complex products in involution semigroups}

\author{Bertalan Bodor}
\author{Erkko Lehtonen}
\author{Thomas Quinn-Gregson}
\author{Nikolaas Verhulst}

\address[B. Bodor, E. Lehtonen, T. Quinn-Gregson, N. Verhulst]%
   {Technische Universit\"at Dresden \\
    Institut f\"ur Algebra \\
    01062 Dresden \\
    Germany}

\address[E. Lehtonen]%
   {Centro de Matem\'atica e Aplica\c{c}\~oes \\
    Faculdade de Ci\^encias e Tecnologia \\
    Universidade Nova de Lisboa \\
    Quinta da Torre \\
    2829-516 Caparica \\
    Portugal}
\address[N. Verhulst]
	{Delft Institute of Applied Mathematics\\
	Faculty of Electrical Engineering, Mathematics \& Computer Science\\
	TU Delft, Building 36 \\
    Mekelweg 4, \\
    2628 CD Delft \\
    The Netherlands}

\date{\today}

\begin{abstract}
When does the complex product of a given number of subsets of a group generate the same subgroup as their union? We answer this question in a more general form by introducing HS\hyp{stability} and characterising the HS\hyp{}stable involution subsemigroup generated by a subset of a given involution semigroup. We study HS\hyp{}stability for the special cases of regular ${}^{*}$\hyp{}semigroups and commutative involution semigroups.
\end{abstract}



\maketitle


\section{Motivation}
\label{sec:motivation}

The direct inspiration for this paper was the following question:
\begin{problem}\label{originalproblem}
Let $G$ be a group. Which subsets $S$ of $G$ satisfy $\gen{S^{-1} S} = \gen{S}$?
\end{problem}
This question arose naturally in the context of invariance groups, minors, and reconstructibility of certain multivariate functions
(see Proposition 4.4.15 and Problem 4.6.1 in \cite{Lehtonen-thesis} and Problem 7.2, Lemma 6.2, and Section 6 in \cite{Lehtonen-2019}).

Of course, it is clear that the inclusion $\gen{S^{-1} S} \subseteq \gen{S}$ always holds, but the converse does not, as the following example shows:

\begin{example}
Let $S_n$ be the symmetric group of degree $n\geq 2$ and let $S$ be a nonempty subset of $S_n$ that contains only odd permutations.
Since the inverse of an odd permutation is odd, $S^{-1} S$ contains only even permutations, so $\gen{S^{-1} S}$ must be a subgroup of the alternating group $A_n$.
However, $\gen{S}$ contains also odd permutations because the generators are odd, so the inclusion $\gen{S^{-1} S} \subset \gen{S}$ must be proper.
\end{example}

We found that this problem has the following rather satisfying solution (which will turn out to be an immediate consequence of Proposition~\ref{characterisation}):
\begin{corollary}\label{originalanswer}
A nonempty subset $S$ of a group $G$ satisfies $\gen{S^{-1} S} \neq \gen{S}$ if and only if $S$ is contained in a nontrivial left coset of a proper subgroup of $G$, i.e., there exists a $g \in G$ and a proper subgroup $G'$ of $\gen{S}$ such that $S \subseteq gG'$.\footnote{Recall that the \emph{trivial coset} of a subgroup is that subgroup itself. It is a well\hyp{}known fact that every left coset of a proper subgroup of a group is also a right coset of some proper subgroup. Therefore we could also write \emph{right coset} instead of \emph{left coset}.}
\end{corollary}

We quickly realised that the methods needed to solve this problem make but little use of the properties of groups, so we turned our attention to the following very natural generalisation of the original problem:
\begin{problem}\label{generalproblem}
Let $S$ be an involution semigroup. For which subsets $S_1, \dots, S_n$ of $S$ do we have $\gen{S_1S_2 \cdots S_n} = \gen{S_1 \cup S_2 \cup \cdots \cup S_n}$?
\end{problem}
In this form, the problem proved too hard for us. It turns out that the vital ingredient which makes Problem \ref{originalproblem} doable and Problem \ref{generalproblem} very hard is \emph{HS\hyp{}stability} (which we will introduce in Section \ref{grouplike}). If, in Problem \ref{generalproblem}, we consider not the involution semigroup generated by $S_{1}, \dots, S_{n}$ but the \emph{HS\hyp{}stable} involution semigroup generated by $S_{1},\dots,S_{n}$, a characterisation very much like that in Corollary \ref{originalanswer} is possible (cf.\ Proposition \ref{characterisation}). It will turn out that, for groups, HS\hyp{}stability is a trivial concept and our general characterisation will yield Corollary \ref{originalanswer} as a special case.

The later sections of this paper are dedicated to an attempt at understanding the concept of HS\hyp{}stability.
We obtain a necessary condition for an involution semigroup to be \emph{HS\hyp{}simple,} i.e., for it to have no proper HS\hyp{}stable involution subsemigroup, in terms of group morphic images (see Section~\ref{sec:trivial}).
We characterise the HS\hyp{}stable involution semigroup generated by a given subset of an involution semigroup (see Section~\ref{sec:genHS}) and showcase this result in the particular cases of regular ${}^{*}$\hyp{}semigroups and commutative involution semigroups (see Sections~\ref{sec:regular} and \ref{sec:commutative}, respectively).

We conclude the paper with a brief coda,  Section~\ref{sec:semilattice}, in which we consider Problem~\ref{generalproblem} for semilattices.
The tools of Section~\ref{grouplike} are of little use when dealing with semilattices because they are HS\hyp{}simple (see Corollary~\ref{cor:idempgen}), but the problem can be solved without much difficulty in a different way.


\section{Preliminaries}
\label{sec:preliminaries}

We assume the reader is familiar with the fundamentals of semigroup theory.
In this section we recall a few notions that will be used throughout the paper.
For general background and additional information on semigroups, we refer the reader to the monograph of Howie~\cite{Howie}.

If $(S, {\circ})$ is a semigroup and $A_1, \dots, A_n$ are subsets of $S$, we define the \emph{complex product} of the subsets as
\[
A_1 \cdots A_n := \{a_1 \cdots a_n \mid a_i \in A_i: \, 1 \leq i \leq n\}.
\]
where, as is usual, we have denoted the binary operation $\circ$ simply by juxtaposition. If, in particular, there is some $A\subseteq S$ with $A_{i}=A$ for all $1\leq i\leq n$ we write $A^{n}:=A_{1}\cdots A_{n}$. In this context, we will often denote a singleton by its single element.
For example, given a subgroup $H$ of a group $G$
and an element $g \in G$, the complex products $\{g\} H$ and $H \{g\}$ will be written simply as $gH$ and $Hg$, respectively; this coincides, both in notation and meaning, with the left and right cosets of $H$ in $G$ with respect to $g$.

If $(S, {\circ})$ is a semigroup and ${}^*$ is an involution, i.e., a unary operation ${}^{*} \colon S \to S$ for which the identities 
\[
  (x^*)^* = x \quad \text{and} \quad (xy)^* = y^* x^*
\] 
hold for all $x, y \in S$, we call $(S, {\circ}, {}^{*})$ an \emph{involution semigroup}. If $(G, {\circ}, {}^{-1}, 1)$ is a group, then $(G, {\circ}, {}^{-1})$ is clearly an involution semigroup.
Less trivially, the set of  $n\times n$    matrices over the complex numbers forms an involution semigroup with the natural multiplication and conjugate transposition as involution.

Let $(S, {\circ}, {}^*)$ be an involution semigroup.
A subset $T$ of $S$ is called an \emph{involution subsemigroup} if $T$ is closed under $\circ$ and ${}^*$, so that
\[
   x,y \in T \implies xy \in T \quad\text{and} \quad x \in T \implies x^* \in T. 
\] 
Given $T \subseteq S$, we denote by $\gen{T}$ the involution subsemigroup generated by $T$, i.e., the smallest involution subsemigroup of $S$ containing $T$.
It is well known that, if $S$ is a group, $\gen{T}$ is the subgroup of $S$ generated by $T$.

If $(S, {\circ}, {}^*)$ is an involution semigroup and $A \subseteq S$, we will write
\[
A^* := \{a^* \mid a \in A\}.
\]


\section{HS\hyp{}stability and the original problem}
\label{grouplike}

From now on, unless indicated otherwise, $S$ will always denote a semigroup which is endowed with an involution ${}^{*}$.

\begin{definition}
We call an element of the form $x x^{*}$ for some $x \in S$ a \emph{hermitian square,} and we let $H_S := \{x x^* \mid x \in S\}$ be the set of all hermitian squares of $S$.
An involution subsemigroup $T$ of $S$ is called \emph{HS\hyp{}stable} if 
\begin{enumerate}[label=(HS:\arabic*),labelwidth=1cm,labelsep=1cm,leftmargin=2cm,itemsep=1ex]
\item\label{HS1} $H_S \subseteq T$,
\item\label{HS2} $\forall h \in H_S \, \forall x, y \in S \colon xhy \in T \implies xy \in T$.
\end{enumerate}
For any subset $B \subseteq S$,
we denote by $\genHS{B}$ the smallest HS\hyp{}stable involution subsemigroup of $S$ containing $B$ and say that $\genHS{B}$ is \emph{generated} by $B$.
\end{definition} 

Note that $\genHS{B}$ is well defined.
Indeed, the whole involution semigroup $S$ is always HS\hyp{}stable and the intersection of HS\hyp{}stable involution subsemigroups is again HS\hyp{}stable, so $\genHS{B}$ is just the intersection of all HS\hyp{}stable involution subsemigroups containing $B$.

\begin{lemma}\label{specialform}
For any nonempty subsets $S_1, \dots, S_n$ of an involution semigroup $S$ and any $1 \leq k \leq n$ we have
\[
S_1 \cdots S_k S_k^{*} \cdots S_1^{*} \subseteq \genHS{S_1 \cdots S_n}.
\]
\end{lemma}

\begin{proof}
Pick $x = s_1 \cdots s_k t_k^{*} \cdots t_1^{*}$ in the set on the left (with $s_i, t_i \in S_i$, $1 \leq i \leq k$) and fix some $g_j \in S_j$ for all $k < j \leq n$.
Put $y := g_{k+1} \cdots g_n$.
Then
\begin{align*}
s_1 \cdots s_k y y^* (t_1 \cdots t_k)^* &= (s_1 \cdots s_k y)(t_1 \cdots t_k y)^* \\
&= (s_1 \cdots s_k g_{k+1} \cdots g_n)(t_1 \cdots t_k g_{k+1} \cdots g_n)^{*}
\end{align*}
which is in
\begin{equation*}
(S_1 \cdots S_n)(S_1 \cdots S_n)^* \subseteq \gen{S_1 \cdots S_n} \leq \genHS{S_1 \cdots S_n}.
\end{equation*}
Since $y y^*$ is a hermitian square and $\genHS{S_1 \cdots S_n}$ is HS\hyp{}stable,
\[
x = (s_1 \cdots s_k)(t_1 \cdots t_k)^* \in \genHS{S_1 \cdots S_n}
\]
follows.
\end{proof}

 We want to know when $\genHS{S_1 S_2 \cdots S_n} = \genHS{S_1 \cup S_2 \cup \dots \cup S_n}$. Taking $S' := \genHS{S_1 \cup S_2 \cup \dots \cup S_n}$ in the following proposition  provides the answer.

\begin{proposition}\label{characterisation}
Let $S$ be an involution semigroup.
For any nonempty subsets $S_1, \dots, S_n$ of $S$ and for any involution subsemigroup $S'$ of $S$, the following are equivalent:
\begin{enumerate}[label={\upshape(\arabic*)}]
	\item\label{characterisation:1} $\genHS{S_1 \cdots S_n} \lneq S'$.
	\item\label{characterisation:2} There exists an HS\hyp{}stable $T \lneq S'$ and $a_1, \dots, a_{n-1} \in S$ with $a_{i-1} S_i a_i^{*} \subseteq T$ for all $1 < i < n$ as well as $S_1 a_1^{*} \subseteq T$ and $a_{n-1} S_n \subseteq T$.
\end{enumerate}
\end{proposition}
\begin{proof}
\ref{characterisation:2} $\boxed{\Rightarrow}$ \ref{characterisation:1}:
Let $T$ and $a_{1}, \dots, a_{n-1}$ be as in \ref{characterisation:2}.
Then
\[
S_1 a_1^{*} a_1 S_2 a_2^{*} \cdots a_{n-2} S_{n-1} a_{n-1}^{*} a_{n-1} S_n \subseteq T^n \subseteq T.
\]
Since all $a_i^{*} a_{i}$ are hermitian squares and $T$ is HS\hyp{}stable, we can conclude from \ref{HS2} that any element of $S_1 \cdots S_n$ must also be in $T$, so $\genHS{S_1 \cdots S_n} \subseteq T$.

\ref{characterisation:1} $\boxed{\Rightarrow}$ \ref{characterisation:2}:
Put $T = \genHS{S_1 \cdots S_n} \lneq S'$.
Fix $x_i \in S_i$ for all $1 \leq i \leq n$, and set $a_i := x_1 \cdots x_i$ and $h_i := a_i^{*} a_i \in H_S$, for all $1 \leq i \leq n$.
By Lemma~\ref{specialform}, we have
\[
h_{i-1} y h_i = a_{i-1}^{*} (x_1 \cdots x_{i-1} y)(x_i^{*} \cdots x_1^{*}) a_i \in a_{i-1}^{*} T a_i
\]
for all $y \in S_i$ with $1 < i < n$.
Consequently,
\[
a_{i-1} h_{i-1} y h_{i} a_{i}^{*}
\in a_{i-1} a_{i-1}^{*} T a_i a_i^{*}
= h_{i-1}^{*} T h_{i}^{*}
\subseteq H_S T H_S
\subseteq T,
\]
which implies $a_{i-1} y a_i^{*} \in T$ because $h_{i-1}$ and $h_i$ are hermitian squares and $T$ is HS\hyp{}stable.
If $y \in S_1$, then applying Lemma~\ref{specialform} with $k = 1$ gives
\[
y a_1^{*} \subseteq T
\]
and if $y \in S_n$ we find
\[
a_{n-1} y = x_1 \cdots x_{n-1} y \in S_1 \cdots S_n \subseteq T.
\qedhere
\]
\end{proof}

As explained in Section~\ref{sec:preliminaries}, any group $G$ is an involution semigroup with the inverse operation ${}^{-1}$ as the involution, and the involution subsemigroups of $G$ are just the subgroups.
The only hermitian square is then the neutral element and every subgroup is HS\hyp{}stable.
Moreover, the conditions $a_{i-1} S_i a_i^{-1} \subseteq T$, $S_1 a_1^{-1} \subseteq T$, $a_{n-1} S_n \subseteq T$ are equivalent to $S_i \subseteq a_{i-1}^{-1} T a_i$, $S_1 \subseteq T a_1$, $S_n \subseteq a_{n-1}^{-1} T$, respectively.
Proposition~\ref{characterisation} then reduces to the following.

\begin{corollary}
Let $G$ be a group.
For any nonempty subsets $S_1, \dots, S_n$ of $G$ and for any subgroup $G' \leq G$, the following are equivalent:
\begin{enumerate}[label={\upshape(\arabic*)}]
	\item $\gen{S_1 \cdots S_n} \lneq G'$.
	\item There exists a subgroup $T \lneq G'$ and $a_1, \dots, a_{n-1} \in G$ with $S_i \subseteq a_{i-1}^{-1} T a_i$ for all $1 < i < n$ as well as $S_1 \subseteq T a_1$ and $S_n \subseteq a_{n-1}^{-1} T$.
\end{enumerate}
\end{corollary}

In the case when $n = 2$, $S_1 = S_2^{-1}$, and $G' = \gen{S_1 \cup S_2} =\gen{S_1}=\gen{S_2}$, this further reduces to Corollary~\ref{originalanswer}, answering Problem~\ref{originalproblem}.


\section{Group morphic images and HS\hyp{}stability}
\label{sec:trivial}

The characterisation from Section~\ref{grouplike} is useful if the involution semigroup under consideration is ``group\hyp{}like'' in the sense that there are few hermitian squares and therefore many HS\hyp{}stable involution subsemigroups.
However, it might happen that an involution semigroup has very few or indeed no proper HS\hyp{}stable involution subsemigroups at all, in which case Proposition~\ref{characterisation} becomes trivial.
We will say that an involution semigroup is \emph{HS\hyp{}simple} if it has no proper HS\hyp{}stable involution subsemigroups.

In this section we give a necessary condition for the HS\hyp{}simplicity of an involution semigroup.
The terminology surrounding  group morphic images will be vital to our main result (Theorem~\ref{thm: main}) that classifies HS\hyp{}stable involution  subsemigroups.

We will sometimes be interested in an involution semigroup considered only as a semigroup, i.e., we sometimes want to forget about the involution. Consequently, we need to be explicit in our terminology regarding homomorphisms between semigroups and between involution semigroups.

\begin{definition} 
Given involution semigroups $S$ and $T$, a map $\phi \colon S \to T$ is called
a \emph{$({\circ})$\hyp{}homomorphism} if we have $\phi(ab) = \phi(a) \phi(b)$ for all  $a, b \in S$.
If, additionally, $\phi(a^*) = \phi(a)^*$ for all $a \in S$, we call $\phi$ a \emph{$({\circ},{}^{*})$\hyp{}homomorphism.}
We call $T$ a \emph{$({\circ})$\hyp{}morphic image} (respectively \emph{$({\circ},{}^{*})$\hyp{}morphic image}) of $S$ if there is a surjective $({\circ})$\hyp{}homomorphism (respectively $({\circ},{}^{*})$\hyp{}homomorphism) from $S$ to $T$.
\end{definition}

\begin{lemma}\label{lemma:preimage} 
Let $\phi \colon S\to S'$  be a $({\circ},{}^{*})$\hyp{}homomorphism between   involution semigroups $S$ and $S'$, and let $T'$ be an HS\hyp{}stable involution subsemigroup of $S'$.
 Then $\phi^{-1}(T')=\{x\in S \mid \phi(x)\in T'\}$  is an HS\hyp{}stable involution subsemigroup of $S$.
 \end{lemma} 
 \begin{proof}
 Let $T=\phi^{-1}(T')$ and $x,y\in S$, $zz^*\in H_S$. 
 By the general properties of homomorphic preimages, $T$ is
an involution subsemigroup of $S$. Moreover, 
 \[ \phi(zz^*)=\phi(z)\phi(z^*)=\phi(z)\phi(z)^* \in H_{S'}\subseteq T',
 \]
 and so $zz^*\in T$. Hence \ref{HS1} holds. 
 Now suppose $xzz^*y\in T$, so that 
 \[ \phi(xzz^*y)=\phi(x)\phi(zz^*)\phi(y)\in T'. 
 \]
 Since $T'$ is HS\hyp{}stable and $\phi(zz^*)\in H_{S'}$ we have $\phi(x)\phi(y)=\phi(xy)\in T'$, so that $xy\in T$. Hence \ref{HS2} holds. 
\end{proof}  

\begin{corollary} \label{cor:preimage-simple}
An involution semigroup $S$ is HS\hyp{}simple if and only if every $({\circ},{}^{*})$\hyp{}homomorphic image of $S$ is HS\hyp{}simple. 
\end{corollary} 
\begin{proof}
$\boxed{\Rightarrow}$
Let $S'$ be a $({\circ},{}^{*})$\hyp{}homomorphic image of $S$, say  $\phi \colon S \twoheadrightarrow S'$. 
If $T'$ is an HS\hyp{}stable involution subsemigroup of $S'$ then $\phi^{-1}(T')=S$ by Lemma \ref{lemma:preimage} as $S$ is HS\hyp{}simple.
Hence $T'\supseteq \phi(\phi^{-1}(T'))=\phi(S)=S'$, and thus $T'=S'$\\ 
 $\boxed{\Leftarrow}$
Immediate, as $S$ is the $({\circ},{}^{*})$\hyp{}homomorphic image of the identity map.
\end{proof}

Of particular importance are \emph{group $({\circ})$-} and \emph{$({\circ},{}^{*})$\hyp{}morphic images,} that is, $({\circ})$- and $({\circ},{}^{*})$\hyp{}morphic images, respectively, that are groups.
Group $({\circ})$\hyp{}morphic images are well understood (see \cite{Gigon}).

\begin{remark}
\label{rem:2hom-inv}
Notice that if $\phi \colon S \to G$ is a $({\circ})$\hyp{}homomorphism between an involution semigroup and a group $(G, {\circ}, {}^{-1},1)$, then it preserves the involution if and only if $H_S \subseteq \phi^{-1}(1) = \{s \in S \mid \phi(s) = 1\}$.
This follows from the fact that 
\[
\phi(a^{*}) = \phi(a)^{-1} \quad\iff\quad \phi(a a^{*}) = \phi(a) \phi(a^{*}) = 1. 
\]
\end{remark}

The following pair of corollaries are immediate from Lemma~\ref{lemma:preimage} and Corollary~\ref{cor:preimage-simple}, since the trivial subgroup of a group is HS\hyp{}stable and hence no nontrivial group is HS\hyp{}simple.  

\begin{corollary}\label{lemma:inv map} 
Let $S$ be an involution semigroup and let $G$ be a group.
If $G$ is a $({\circ},{}^{*})$\hyp{}morphic image of $S$, say $\phi \colon S \twoheadrightarrow G$, then $\phi^{-1}(1)$ is an HS\hyp{}stable involution subsemigroup of $S$. 
\end{corollary} 

\begin{corollary}\label{corollary:simplenogroup}
If $S$ is an HS\hyp{}simple involution semigroup, then it has only trivial group $({\circ},{}^{*})$\hyp{}morphic images.
\end{corollary}

\begin{remark} The converse implication of Corollary \ref{corollary:simplenogroup} does not hold in general.
For example, if $S$ is an involution semigroup with a zero element $0$ and with $S \neq S^{2}$, then $S$ is not HS\hyp{}simple since $S^{2}$ is an HS\hyp{}stable proper subsemigroup by Lemma~\ref{lemma:commHS}.
However, the only possible group $({\circ},{}^{*})$\hyp{}morphic image of $S$ is the trivial one.
Indeed, assume $\phi$ is a $({\circ},{}^{*})$\hyp{}morphism from $S$ onto a group $G$.
It must map $0$ to $1$ and for any $g \in G$ there must be some $s_{g} \in S$ with $\phi(s_{g}) = g$.
Consequently,
\[
g = g1 = \phi(s_{g}) \phi(0) = \phi(s_{g}0) = 1
\]
for all $g \in G$; hence $G$ is trivial.
\end{remark} 

Next we give a necessary and sufficient condition for  an involution subsemigroup to equal $\phi^{-1}(1)$ for some surjective $({\circ},{}^{*})$\hyp{}morphism $\phi$  onto a group.  

For a subset $T$ of $S$, we define $T \omega := \{s \in S \mid \exists t \in T \colon st \in T\}$ and call it the \emph{closure} of $T$ (in $S$).

\begin{remark}
\label{rem:closure}
By the following lemma, $\omega$ is monotone in general and extensive over subsemigroups.
However, it is not necessarily idempotent over involution subsemigroups.
In fact, we will show in Example~\ref{example:regular} that there exists an involution subsemigroup $T$ with $T \omega \subsetneq (T \omega) \omega$. However, we will show in Corollary~\ref{cor:closedproperty2} that $\omega$ becomes a closure operator when restricted to involution subsemigroups $T$ satisfying a particular conjugation condition.
\end{remark}

\begin{lemma}\label{lemma:closureproperty}
Let $S$ be an involution semigroup with subsets $T$ and $T'$.
Then  
\[
T \subseteq T' \quad \implies \quad T \omega \subseteq T' \omega. 
\] 
Moreover, if $T$ forms a subsemigroup of $S$ then $T \subseteq T \omega$ and $T \omega \subseteq (T \omega) \omega$. 
\end{lemma} 

\begin{proof}
Assume $T \subseteq T'$, and let $s \in T \omega$, so $st, t \in T$ for some $t \in T$.
Hence $st, t \in T'$, so $s \in T' \omega$, and we conclude that $T \omega \subseteq T' \omega$.

Suppose now that $T$ is a subsemigroup of $S$, and let $t \in T$. Then $tt, t \in T$ so that $t \in T \omega$; hence $T \subseteq T \omega$.
The inclusion $T \omega \subseteq (T \omega) \omega$ follows then immediately from the first result.
\end{proof}

\begin{remark}
Note that, if we drop the condition that $T$ forms a subsemigroup of $S$ in the second statement above, then $T$ need not be contained in $T \omega$.
For example, consider again the symmetric group $S_n$ of degree $n \geq 2$, and let $T$ be a nonempty subset that contains only odd permutations.
Then $T \omega$ contains only even permutations.
\end{remark} 

We let $E_S$ denote the set of idempotents of $S$.
We call a subset $T$ of $S$
\begin{itemize}
\item \emph{full} if $E_S \subseteq T$, 
\item \emph{closed} if $T = T \omega$, 
\item \emph{reflexive} if $ab \in T$ implies $ba \in T$,
\item \emph{dense} if for all $s \in S$ there exist $x, y \in S$ with $sx, ys \in T$. 
\end{itemize}

Let $T$ be a subsemigroup of $S$.
It follows from \cite[Theorem 2.4]{Gigon} that $T = \phi^{-1}(1)$ for some surjective $({\circ})$\hyp{}homomorphism $\phi \colon S \twoheadrightarrow G$ with $G$ a group if and only if $T$ is full, closed, reflexive, and dense.
For involution semigroups this result becomes the following: 

\begin{lemma}\label{cor:morphic} 
Let $S$ be an involution semigroup with an involution subsemigroup $T$.
Then $T = \phi^{-1}(1)$ for some surjective $({\circ},{}^{*})$\hyp{}homomorphism $\phi \colon S \twoheadrightarrow G$ with $G$ a group if and only if $T$ is closed and reflexive and $H_S \subseteq T$.
\end{lemma} 

\begin{proof} 
$\boxed{\Rightarrow}$
Let $T = \phi^{-1}(1)$ for some surjective $({\circ},{}^{*})$\hyp{}homomorphism $\phi$ from $S$ to a group $G$.
Then $\phi$ is a $({\circ})$\hyp{}homomorphism, and thus $T$ is closed and reflexive.
Since $\phi$ also preserves ${}^{*}$ we have $H_S \subseteq T$ by Remark~\ref{rem:2hom-inv}. 

$\boxed{\Leftarrow}$
It suffices to show that $T$ is dense and full.
If $s \in S$ then $s s^*, s^* s \in H_S \subseteq T$, so $T$ is dense. If $e \in E_S$ then $e(ee^*) = ee^*$, from which it follows that $E_S \subseteq H_S \omega$.
Since $T$ is closed we thus have 
\[
E_S \subseteq H_S \omega \subseteq T \omega = T
\] 
and so $T$ is full.
Hence there exists a $({\circ})$\hyp{}homomorphism $\phi \colon S \twoheadrightarrow G$ with $T = \phi^{-1}(1)$.
By Remark~\ref{rem:2hom-inv} the map $\phi$ preserves ${}^*$ since $H_S \subseteq \phi^{-1}(1)$. 
\end{proof}


\section{Finding the HS\hyp{}stable involution subsemigroup generated by a set}
\label{sec:genHS}

If $S$ is an involution semigroup and $T$ is an HS\hyp{}stable involution subsemigroup then the condition $x T x^* \subseteq T$ need not hold for all $x \in S$. For example, if $S$ is a group then all subgroups are HS\hyp{}stable, but non\hyp{}normal subgroups do not satisfy $x T x^{-1} \subseteq T$ for all $x \in S$.
We show in this section that a weakening of this condition together with a weakened closure condition is equivalent to HS\hyp{}stability.
We first require a couple of lemmas.

\begin{lemma}\label{lemma:uss}
Let $S$ be an involution semigroup and $T$ be an HS\hyp{}stable involution subsemigroup of $S$.
Then
\begin{enumerate}[label={\upshape(\roman*)}]
\item\label{lemma:uss:1} $E_S \subseteq T$.
\item\label{lemma:uss:2} $x H_S^2 x^* \subseteq T$ for each $x \in S$.  
\end{enumerate} 
\end{lemma} 
 
\begin{proof} 
\ref{lemma:uss:1}
If $e \in E_S$ then $e^* = (ee)^* = e^* e^*$; hence $e^* \in E_S$.
By \ref{HS1} we have $ee^*, e^*e \in T$, and so 
\[
(ee^*)(e^*e) = ee^*e = (ee)e^*e = e(ee^*)e \in T,
\] 
so by \ref{HS2} we have $e = ee \in T$. 

\ref{lemma:uss:2}
Let $x \in S$ and $a = g g^* h h^* \in H_S^2$ be arbitrary. 
Then $x a a^* x^* = (xa)(xa)^* \in H_S \subseteq T$ by \ref{HS1}.
On the other hand we have
\[
x a a^* x^* = x a (gg^* hh^*)^* x^* = x a (h h^*)(g g^*) x^*,
\]
and so by applying \ref{HS2} to the bracketed hermitian squares we obtain $x a x^* \in T$ as required.
\end{proof}

\begin{corollary}
\label{cor:idempgen}
If $S = \gen{E_S}$, then $S$ is HS\hyp{}simple.
\end{corollary}

\begin{proof}
Assume $S = \gen{E_S}$, and let $T$ be an HS\hyp{}stable involution subsemigroup of $S$.
Then by Lemma~\ref{lemma:uss}\ref{lemma:uss:1} we have $E_S \subseteq T$, and so $S = \gen{E_S} \subseteq T$. 
\end{proof}

\begin{lemma}\label{lemma:fullclosed sub}
Let $S$ be an involution semigroup and let $T \subseteq S$ be such that
\begin{enumerate}[label={\upshape(\arabic*)}]
\item\label{lfs:1} $xH_S^2 x^*\subseteq T$ for each $x\in S$;
\item\label{lfs:2} $T \omega \cap S^2 = T \cap S^2$;  
\item\label{lfs:3} $T \setminus S^2 = T^* \setminus S^2$. 
\end{enumerate}
Then $T$ forms an involution subsemigroup of $S$ containing $H_S$.  
\end{lemma}

\begin{proof}
Assume $T$ satisfies \ref{lfs:1}, \ref{lfs:2}, and \ref{lfs:3}; we first show that $H_S\subseteq T$. Let $g g^* \in H_S$. Then $g g^* (g g^* g g^* g g^*)$ and $g g^* g g^* g g^*$ are both elements of $g H_S^2 g^*$, and thus of $T$ by \ref{lfs:1}. 
Hence $g g^* \in T \omega$, and so $gg^*\in T$ by \ref{lfs:2}.

Now let $x \in T$.
If $x \notin S^2$ then we immediately get $x^* \in T$ by \ref{lfs:3}.
Suppose instead that $x = yz \in S^2$.
Then $x^* = z^* y^* \in S^2$, and
\[
x^* x x^* x = (x^* x)(x^* x)^* \in H_S \subseteq T
\]
by \ref{lfs:1}.
Since $x, x^* x x^* x \in T$, we have $x^* x x^* \in T \omega$.
Clearly $x^* x x^* \in S^2$, so $x^* x x^* \in T \omega \cap S^2 = T \cap S^2$ by \ref{lfs:2}.
Since $x x^* \in T$ we have $x^* \in T \omega$, and as $x^* \in S^2$ we get $x^* \in T$ by \ref{lfs:2}. 

Now suppose $x, y \in T$.
Then, as $(x y y^{*}) x^{*} = (xy)(xy)^{*} \in H_S \subseteq T$ and $x^* \in T$, we have that $x y y^{*} \in T \omega \cap S^2$, and so $x y y^* \in T$ by \ref{lfs:2}.
Similarly, $xy \in T \omega \cap S^2$ as $y^{*} \in T$ and so $xy \in T$ by \ref{lfs:2}.
Hence $T$ is an involution subsemigroup. 
\end{proof}

\begin{corollary} \label{cor:closedproperty2}
Let $S$ be an involution semigroup with involution subsemigroup $T$.
If $x H_S^2 x^* \subseteq T$ for each $x \in S$, then $T \omega$ is a closed involution subsemigroup containing $H_S$. 
\end{corollary} 

\begin{proof}
We first show that $T \omega = (T \omega) \omega$. Since $T$ is a subsemigroup of $S$ it follows from Lemma \ref{lemma:closureproperty} that $T \omega \subseteq (T \omega) \omega$.
For the converse inclusion, let $s \in (T \omega) \omega$, so there exists a $t \in T \omega$ such that $st \in T \omega$.
This in turn implies that there exist $v, w \in T$ such that $stv, tw \in T$.
Then 
\[
(stv) (v)^* (w) (tw)^* = s (t v v^* w w^* t^*) \in T
\] 
and $t v v^* w w^* t^* \in t H_S^2 t^* \subseteq T$.
Hence $s \in T \omega$ and $T\omega$ is therefore closed.

We now show that $T \omega$ satisfies the conditions of Lemma~\ref{lemma:fullclosed sub}.
 Condition \ref{lfs:1} follows from our hypothesis, since $T\subseteq T\omega$ by Lemma \ref{lemma:closureproperty}. 
Condition \ref{lfs:2} follows immediately from the closedness of $T \omega$.

The last condition, \ref{lfs:3}, follows immediately if we show that $T \omega = (T \omega)^*$.
Let $s \in T \omega$.
Then there exists a $t \in T$ such that $st \in T$.
By our hypothesis, $s^* t t^* t t^* s \in T$.
Since $t, t^*, st \in T$, we also have $s^* t t^* t t^* s t, t t^* t t^* s t \in T$.
This implies $s^* \in T \omega$; hence $(T \omega)^* \subseteq T \omega$.
Moreover, $T \omega = (T \omega)^{**} \subseteq (T \omega)^*$.
\end{proof}

\begin{lemma}\label{lemma:commHS}
Let $S$ be an involution semigroup and $T$ an involution subsemigroup of $S$.
Then $T$ is HS\hyp{}stable if and only if $T \cap S^2$ is HS\hyp{}stable.
In particular, $S^2$ is HS\hyp{}stable. 
\end{lemma} 

\begin{proof}
Notice that $H_S \subseteq S^2$, that $x g g^* y \in T$ if and only if $x g g^* y \in T \cap S^2$, and that $xy \in T$ if and only if $xy \in T \cap S^2$, from which the first result follows.
The second statement follows by noting that $S^{2}$ is an involution subsemigroup and taking $T = S$.
\end{proof}

\begin{theorem} \label{thm: main}
Let $S$ be an involution semigroup and $T \subseteq S$.
Then $T$ is an HS\hyp{}stable involution subsemigroup if and only if
\begin{enumerate}[label={\upshape(\arabic*)}]
\item\label{main:1} $xH_S^2 x^*\subseteq T$ for each $x\in S$; 
\item\label{main:2} $T\omega \cap S^2 = T\cap S^2$; 
\item\label{main:3} $T\setminus S^2 = T^* \setminus S^2$. 
\end{enumerate}
\end{theorem}

\begin{proof}
$\boxed{\Rightarrow}$
Let $T$ be an HS\hyp{}stable involution subsemigroup, so condition \ref{main:1} holds by Lemma~\ref{lemma:uss}\ref{lemma:uss:2}, and \ref{main:3} is immediate because $T = T^*$.

It remains to show \ref{main:2}.
The inclusion $T \cap S^2 \subseteq T \omega \cap S^2$ is immediate because $T \subseteq T \omega$ holds by Lemma~\ref{lemma:closureproperty}.
For the converse inclusion, let $s \in T \omega \cap S^2$, say $s = xy$ and $st, t \in T$.
Then, as $T$ is an involution subsemigroup, we have $t^* \in T$ and so 
\[
x (y t t^* y^* x^* x) y = (xyt) t^* (y^* x^* x y) = (xyt) t^* (y^* x^*) (y^* x^*)^* \in T^2 H_S \subseteq T  
\]
by \ref{HS1}.
However, $y t t^* y^* x^* x = (yt) (yt)^* (x^* x) \in H_S^2 \subseteq T$, and so $s = xy \in T$ by two applications of \ref{HS2}.
Hence $T \omega \cap S^2 \subseteq T \cap S^2$.

$\boxed{\Leftarrow}$ Let $T$ satisfy \ref{main:1}, \ref{main:2}, and \ref{main:3}. Then $T$ forms an involution subsemigroup of $S$ containing $H_S$ by Lemma~\ref{lemma:fullclosed sub}, so \ref{HS1} holds. 

Now let $x g g^* y \in T$, so $(x g g^* y)^* \in T$ as $T$ is closed under ${}^*$.
Then
\[
xy (x g g^* y)^* = x y y^* g g^* x^* \in x H_S^2 x^* \subseteq T 
\] 
and $xy \in T \omega \cap S^2$, so that $xy \in T$ by \ref{main:2}.
Hence \ref{HS2} holds, and $T$ is HS\hyp{}stable. 
\end{proof}

\begin{theorem}\label{thm: main_gen} 
Let $S$ be an involution semigroup and let $A \subseteq S$.
Then 
\[
\genHS{A} = (\gen{A \cup \bigcup_{x \in S} x H_S^2 x^*} \omega \cap S^2) \cup ((A \cup A^*) \setminus S^2).
\]
\end{theorem}

\begin{proof}
Let $A' := \gen{A \cup \bigcup_{x \in S} x H_S^2 x^*}$.
We first show that
\[
K := (A' \omega \cap S^2) \cup ((A \cup A^*) \setminus S^2)
\]
forms an HS\hyp{}stable involution subsemigroup that contains $A$.
Observe first that $A \subseteq A' \subseteq A' \omega$ holds by Lemma~\ref{lemma:closureproperty} since $A'$ is an involution subsemigroup of $S$.
Therefore,
\[
A = (A \cap S^2) \cup (A \setminus S^2) \subseteq (A' \omega \cap S^2) \cup ((A \cup A^*) \setminus S^2) = K.
\]

We check the conditions \ref{main:1}--\ref{main:3} from Theorem~\ref{thm: main}.
For each $x \in S$, we have $x H_S^2 x^* \subseteq A' \cap S^2 \subseteq A' \omega \cap S^2 \subseteq K$, so condition \ref{main:1} holds.
Condition \ref{main:3} follows from the facts that $K \setminus S^2 = (A \cup A^*) \setminus S^2$ and $x \in S^2$ if and only if $x^* \in S^2$.
In order to prove \ref{main:2}, notice that $K \subseteq A' \omega$.
Indeed, $A \subseteq A' \subseteq A' \omega$ and, as $A' \omega$ satisfies \ref{main:1}, it follows by Corollary~\ref{cor:closedproperty2} that $A' \omega$ is a closed involution subsemigroup of $S$, and in particular contains $A^*$.
Consequently,
\[
K \omega \cap S^2 \subseteq (A' \omega) \omega \cap S^2 = A' \omega \cap S^2 = K \cap S^2. 
\]
In order to prove the converse inclusion, let $x \in K \cap S^2 = A' \omega \cap S^2$.
Since $A' \omega$ is an involution subsemigroup of $S$, we have $x^2 \in A' \omega$ and clearly $x^2 \in S^2$, so $x^2 \in A' \omega \cap S^2 = K \cap S^2$.
It follows that $x \in (K \cap S^2) \omega \subseteq K \omega$ by the monotonicity of $\omega$, and hence $x \in K \omega \cap S^2$.
Therefore also \ref{main:2} holds, and we conclude that $K$ is an HS\hyp{}stable subsemigroup containing $A$; hence $\genHS{A} \subseteq K$.

It remains to show that $K \subseteq \genHS{A}$.
Since $\genHS{A}$ is an HS\hyp{}stable involution subsemigroup containing $A$ we have $A \cup \bigcup_{x \in S} x H_S^2 x^* \subseteq \genHS{A}$ by condition \ref{main:1} of Theorem~\ref{thm: main}, and hence $A'$ is an involution subsemigroup of $\genHS{A}$.
Applying the closure operation, and then intersecting with $S^2$ we obtain  
\[
A' \omega \cap S^2 \subseteq \genHS{A} \omega \cap S^2 = \genHS{A} \cap S^2,
\] 
where the first inclusion holds by the monotonicity of $\omega$ and the final equality holds by condition \ref{main:2} of Theorem \ref{thm: main}.
Since $\genHS{A}$ is closed under ${}^*$ we have $A \cup A^* \subseteq \genHS{A}$.
Hence $K \subseteq (\genHS{A} \cap S^2) \cup (\genHS{A} \setminus S^2) = \genHS{A}$.
\end{proof}

If $S=S^2$ then Theorems~\ref{thm: main} and \ref{thm: main_gen} simplify significantly:

\begin{corollary}\label{cor: main}
Let $S$ be an involution semigroup with $S = S^2$ and let $T \subseteq S$.
Then $T$ is an HS\hyp{}stable involution subsemigroup if and only if $T$ is closed and $x H_S^2 x^* \subseteq T$ for each $x \in S$. 
In particular, if $A \subseteq S$ then 
\[
\genHS{A} = \gen{A \cup \bigcup_{x \in S} x H_S^2 x^*} \omega.
\]
\end{corollary} 

\begin{remark} 
We cannot replace $\bigcup_{x \in S} x H_S^2 x^*$ in the result above with the set of conjugates of hermitian squares $\bigcup_{x \in S} x H_S x^*$ of $S$ (or indeed with $H_S$), as we will show in Example~\ref{example:regular}.
In fact, we show that there exists an involution semigroup $S$ such that $S = S^2$ for which $\gen{H_S} \omega$ is equal to $\gen{\bigcup_{x \in S} x H_S x^*} \omega$ but is not HS\hyp{}stable, so 
\[
\genHS{H_S} = \gen{\bigcup_{x \in S} x H_S^2 x^*} \omega \supsetneq \gen{\bigcup_{x \in S} x H_S x^*} \omega = \gen{H_S} \omega. 
\]  
\end{remark} 

We end this section with a quick application of the above results in the case of involution semigroups $(S, {\circ}, {}^{*})$ with a zero element, denoted by $0$.
Notice that $0^{*} s = (s^{*} 0)^{*} = 0^{*} = (0 s^{*})^{*} = s0^{*}$ for every $s \in S$.
Since a semigroup contains at most one absorbing element, $0 = 0^{*}$ follows.

\begin{corollary}\label{cor: isgwithzero}
Let $S$ be an involution semigroup containing a zero element $0$ and $A \subseteq S$.
Then 
\[
\genHS{A} = S^2 \cup ((A \cup A^*) \setminus S^2).  
\] 
Consequently, $S$ is HS\hyp{}simple if and only if $S = S^2$. 
\end{corollary}

\begin{proof}
We first note that for any $B \subseteq S$, if $0 \in B$ then $B \omega = S$.
Indeed, $s0 = 0 \in B$, and so $s \in B \omega$ for any $s \in S$.
Hence as $0 H_S^2 0^* = \{0\}$, it follows by Theorem~\ref{thm: main_gen} that 
\[
\genHS{A} = (S \cap S^2) \cup ((A \cup A^*) \setminus S^2)
\] 
and the result follows. 

It is then immediate that, if $S = S^2$, then $S$ is HS\hyp{}simple.
The converse follows from the fact that $S^2$ is HS\hyp{}stable by Lemma~\ref{lemma:commHS}.
\end{proof}

\begin{corollary}
\label{cor:monoid-zero}
A monoid with involution containing a zero element is HS\hyp{}simple.
\end{corollary} 
\begin{proof}
Since $S$ is a monoid, we have $x=x1\in S^{2}$ for any $x\in S$. Consequently, $S=S^{2}$ and Corollary~\ref{cor: isgwithzero} yields the desired result.
\end{proof}


\section{HS\hyp{}stability for regular \texorpdfstring{${}^{*}$}{*}\hyp{}semigroups} 
\label{sec:regular}

In this section we apply Theorem \ref{thm: main} to an important class of semigroups with involution: regular ${}^{*}$\hyp{}semigroups. 

Given $x \in S$, we call $x' \in S$ an \emph{inverse of $x$} if $x x' x = x$ and $x' x x' = x'$; the set of all inverses of $x$ will be denoted by $V(x)$.
A semigroup $S$ is \emph{regular} if every element has an inverse, and is \emph{orthodox} if further $E_S$ forms a subsemigroup of $S$.
A semigroup is \emph{inverse} if every element $x$ has a unique inverse, which we denote by $x^{-1}$.
The set of idempotents of an inverse semigroup $S$ forms a semilattice, that is, a commutative idempotent semigroup, and hence every inverse semigroup is orthodox.

An involution semigroup $S$ is called a \emph{regular ${}^{*}$\hyp{}semigroup} if $x^* \in V(x)$ for each $x \in S$ (noting that $(S, {\circ})$ forms a regular semigroup).
Semigroups with involution of this type were first studied by Nordahl and Scheiblich in \cite{Nordahl}.
Note that $S = S^2$ for a regular ${}^{*}$\hyp{}semigroup since $x = x (x^* x)$.
Every inverse semigroup forms a regular ${}^{*}$\hyp{}semigroup (with involution ${}^{-1}$), but the converse need not hold as the following example shows.

\begin{example} \label{example:rectangular} 
Let $I$ be a set and define a product on $S = I \times I$ by $(i, j) (k, \ell) = (i, \ell)$.
Then the unary map ${}^{*} \colon S \to S$ given by $(i, j)^{*} = (j, i)$ is an involution, and $V(x) = S$ for each $x \in S$, so that $S$ is a regular ${}^{*}$\hyp{}semigroup.
\end{example} 

A regular ${}^{*}$\hyp{}semigroup $S$ is called an \emph{orthodox} ${}^{*}$\hyp{}semigroup if $(S, {\circ})$ is orthodox.
The example above is clearly an orthodox ${}^{*}$\hyp{}semigroup since $S = E_S$. 

\begin{lemma} \label{H,E}
Let $S$ be a regular ${}^{*}$\hyp{}semigroup. Then
\begin{enumerate}[label={\upshape(\roman*)}]
\item\label{H,E:1} $H_S = \{e \in E_S \mid e^* = e\} \subseteq E_S$, with $H_S = E_S$ if and only if $S$ is inverse. 
\item\label{H,E:2} $H_S^2 = E_S$. 
\end{enumerate} 
Moreover, if $S$ is orthodox then
\begin{enumerate}[label={\upshape(\roman*)},resume]
\item\label{H,E:3} $x e x^* \in E_S$ for each $x \in S$ and $e \in E_S$. 
\end{enumerate}
\end{lemma} 

\begin{proof} 
\ref{H,E:1}
For every $x \in S$ we have $(x x^*) (x x^*) = (x x^* x) x^* = x x^*$.
Hence $x x^* \in E_S$, and $(x x^*)^* = (x^*)^* x^* = x x^*$, so $H_S \subseteq \{e \in E_S \mid e^{*} = e\}$.
For the other containment, let us assume that $e^* = e = e^2$.
Then $e = e e^* \in H_S$.
Therefore we have  $\{e \in E_S \mid e^* = e\}\subseteq H_S$.
The final claim is then immediate from \cite[Lemma 1]{Easdown}. 

\ref{H,E:2}
The fact that $H_S^2 \subseteq E_S$ follows from \ref{H,E:1} and \cite[Theorem 2.5]{Nordahl}.
If $e \in E_S$, then, recalling that $E_S^* = E_S$, we have 
\[
e = e e^* e = e (e^* e^*) e = (e e^*) (e^* e) \in H_S^2. 
\] 

\ref{H,E:3}
Follows from \cite[Proposition 6.2.2]{Howie}.
\end{proof}
 
\begin{definition} 
Given an involution subsemigroup $S$, we let $F_S := \{x e x^* \mid x \in S, e \in E_S\}$. 
\end{definition}

\begin{corollary}\label{thm:conjper} 
Let $S$ be a regular ${}^{*}$\hyp{}semigroup and $T \subseteq S$.
Then $T$ is an HS\hyp{}stable involution subsemigroup of $S$ if and only if $T$ is closed and $F_S \subseteq T$.
Consequently, 
\[
\genHS{T} = \gen{T \cup F_S} \omega. 
\] 
\end{corollary} 

\begin{proof}
Since $H_S^2 = E_S$ by Lemma~\ref{H,E}\ref{H,E:2}, it follows that $F_S = \bigcup_{x \in S} x H_S^2 x^*$.
Hence, as $S = S^2$, the result is immediate from Corollary~\ref{cor: main}.
\end{proof}  

\begin{corollary}
Let $S$ be an orthodox ${}^{*}$\hyp{}semigroup and $T \subseteq S$.
Then $T$ is an HS\hyp{}stable involution subsemigroup of $S$ if and only if $T$ is closed and full.
Consequently, 
\[
\genHS{T} = \gen{T \cup E_S} \omega, 
\]  
and $E_S \omega$ is the minimal HS\hyp{}stable involution subsemigroup of $S$. 
\end{corollary} 

\begin{proof} 
 Since $S$ is orthodox we have $F_S \subseteq E_S$ by  Lemma~\ref{H,E}\ref{H,E:3}. 
The claimed equivalence now follows from Corollary~\ref{thm:conjper} and Lemma~\ref{lemma:uss}\ref{lemma:uss:1}. 
For the second claim, it suffices to show that $E_S = \gen{E_S}$.
This follows from the fact that $S$ is orthodox and that $E_S^* = E_S$ for any involution semigroup.
\end{proof}
  
We note that the corollary above does not hold for general regular ${}^{*}$\hyp{}semigroups.
Indeed, we shall construct a regular ${}^{*}$\hyp{}semigroup $S$ in which $\gen{E_S} \omega$ is not an HS\hyp{}stable involution subsemigroup.

\begin{example}\label{example:regular} 
Let $G$ be a finite group with identity element $e$ and non\hyp{}normal subgroup $K$, that is, there exist $x \in G$ and $a \in K$ with $x a x^{-1} \notin K$.
Let $P$ be an $\mathbb{N} \times \mathbb{N}$ matrix with entries  $p_{i,j}$ ($i, j \in \mathbb{N}$) from $K$ and such that $p_{i,j} = p_{j,i}^{-1}$.
Suppose also $p_{i,1} = p_{1,i} = p_{i,i} = e$ for each $i \in \mathbb{N}$, and $p_{2,3} = a$.
On $S = \mathbb{N} \times G \times \mathbb{N}$, define a product by 
\[
(i, g, j) (k, h, \ell) = (i, g p_{j,k} h, \ell) 
\] 
and involution ${}^*$ by $(i, g, j)^* = (j, g^{-1}, i)$.
Then $S$ forms a regular ${}^{*}$\hyp{}semigroup, called a \emph{Rees matrix involution semigroup} (we refer the reader to \cite{Gerhard} for further information).
We consider $\gen{E_S} \omega$, noting that $\gen{E_S} \subseteq \{(i, h, j) \mid i, j \in \mathbb{N},\, h \in K\}$ by \cite{Howie78}.
Let $(i, g, j) \in \gen{E_S} \omega$, so that there exists $(k, h, \ell) \in \gen{E_S}$ such that
\[
(i, g, j) (k, h, \ell) = (i, g p_{j,k} h, \ell) \in \gen{E_S}. 
\] 
Hence $g p_{j,k} h \in K$, so that $g \in K h^{-1} p_{j,k}^{-1} \subseteq K$ since $h, p_{j,k} \in K$.
Thus $\gen{E_S} \omega \subseteq \{(i, h, j) \mid i, j \in \mathbb{N},\, h \in K\}$.
However $(2, p_{3,2}^{-1}, 3) = (2, a, 3) \in E_S$ and 
\[
(1, x, 1) (2, a, 3) (1, x^{-1}, 1) = (1, x a x^{-1}, 1) \in F_S.
\] 
By Lemma \ref{lemma:closureproperty} we have $F_S\subseteq \gen{F_S} \subseteq \gen{F_S}\omega$, and so $(1,xax^{-1},1)\in \gen{F_S}\omega$. However,  $x a x^{-1} \notin K$ so that $\gen{F_S}\omega$ is not contained in $\gen{E_S} \omega$. Since $\gen{F_S} \omega$ is the minimum HS\hyp{}stable involution subsemigroup of $S$ by Corollary~\ref{thm:conjper}, it follows that $\gen{E_S} \omega$ is not an HS\hyp{}stable involution subsemigroup of $S$ (and thus nor is $\gen{H_S} \omega$).  

This example also allows us to construct an involution subsemigroup $T$ such that $T \omega \neq (T \omega) \omega$, thus showing that $\omega$  in general is not a closure operator on involution subsemigroups of an involution semigroup (see Remark~\ref{rem:closure}).
Given $S$ as above, consider the involution subsemigroup $T = \{(1, e, 1), (1, e, 2), (2, e, 1),\linebreak[1] (2, e, 2)\}$ (note that $T$ is isomorphic to the involution semigroup given in Example~\ref{example:rectangular} with $\card{I} = 2$). 
Then as $p_{3,2} = a^{-1}$ we have 
\[
(2, a, 3) (2, e, 1) = (2, e, 1) \quad \text{and} \quad  (1, e, 3) (1, e, 1) = (1, e, 1), 
\] 
so that $(2, a, 3), (1, e, 3) \in T \omega$.
Also, $(1, a^{-1}, 1) (i, e, j) = (1, a^{-1}, j) \notin T$ for any $i, j \in \mathbb{N}$, and so $(1, a^{-1}, 1) \notin T \omega$.
However, $(1, a^{-1}, 1) (2, a, 3) = (1, e, 3)$ and so $(1, a^{-1}, 1) \in (T \omega) \omega$. Hence $T \omega \neq (T \omega) \omega$.
\end{example}

If $S$ is an orthodox ${}^{*}$\hyp{}semigroup then $E_S \omega = \phi^{-1}(1)$ where $\phi \colon S \twoheadrightarrow G$ is the \textit{greatest group $({\circ})$\hyp{}morphic image of $S$} by \cite[Theorem  4.5]{Gigon}. That is, for every $({\circ})$\hyp{}morphic image of $S$, say $\psi \colon S \twoheadrightarrow H$, there exists a $({\circ})$\hyp{}morphism $\sigma \colon G \rightarrow H$ such that $\psi = \sigma \phi$.
We refer the reader to \cite[Chapter 4]{Gigon} for a further study.  
 
\begin{corollary}
Let $S$ be an orthodox ${}^{*}$\hyp{}semigroup.
Then the following are equivalent:
\begin{enumerate}[label={\upshape(\arabic*)}]
\item\label{cor-star:1} $S$ is HS\hyp{}simple.
\item\label{cor-star:2} Every group $({\circ})$\hyp{}morphic image of $S$ is trivial.
\item\label{cor-star:4} $S=E_S\omega$. 
\end{enumerate}
Moreover, if $S$ is inverse then these are also equivalent to:
\begin{enumerate}[label={\upshape(\arabic*)},resume]
\item\label{cor-star:3} For every $x \in S$ there exists $e \in E_S$ with $e = xe = ex$. 
\end{enumerate}
\end{corollary}

\begin{proof}
Since $E_S \omega = \phi^{-1}(1)$ is the minimal HS\hyp{}stable involution subsemigroup, the equivalence of \ref{cor-star:1}, \ref{cor-star:2} and \ref{cor-star:4} is immediate. 

Now let $S$ be inverse, so that $e^{-1} = e$ for every $e \in E_S$.

\ref{cor-star:3} $\boxed{\Rightarrow}$ \ref{cor-star:1}.
Let $x \in S$, so that there exists $e \in E_S$ with $xe = e$.
Hence $x \in E_S \omega$ and $S=E_{S}\omega$.

\ref{cor-star:1} $\boxed{\Rightarrow}$ \ref{cor-star:3}.
Assume $S = E_S \omega$. Then for any $x \in S$ there exist $e, f \in E_S$ with $xe = f$.
Then $e x^{-1} = f^{-1} = f$, and so $e x^{-1} x e = ff = f$, and hence $e x^{-1} x = f$ as $x^{-1} x \in E_{S}$ and as $E_S$   is commutative.
Consequently, 
\[
xf = x (e x^{-1} x) = x (x^{-1} x e) = xe = f  \quad \text{and} \quad fx = (e x^{-1}) x = f.  
\qedhere
\] 
\end{proof}

\begin{remark} 
If $S$ is an orthodox ${}^{*}$\hyp{}semigroup with $E_S$ forming an HS\hyp{}stable involution subsemigroup, then $E_S = E_S \omega$.
This later condition is a well\hyp{}studied property known as \emph{E\hyp{}unitarity}.
The structure of E\hyp{}unitary regular ${}^{*}$\hyp{}semigroups is given in \cite{Imaoka}.
For example, the free inverse monoid on a set $X$ is $E$\hyp{}unitary, and so $E_S = \{1\}$ is an HS\hyp{}stable involution subsemigroup.  
\end{remark}


\section{HS\hyp{}stability for commutative involution semigroups} 
\label{sec:commutative}

In this section we consider commutative involution semigroups.
Every commutative semigroup comes equipped with an involution, namely the identity map $x^* = x$; such involution semigroups are called \emph{semigroups with trivial involution.} Conversely, every semigroup with trivial involution is clearly commutative. 

For commutative semigroups with trivial involution we have $E_S \subseteq H_S = \{s^2 \mid s \in S\}$.
This fails to hold for general commutative semigroups with involution; take for example the 3\hyp{}element non\hyp{}chain semilattice $Y = \{x, y, 0\}$ with $xy = x0 = y0 = 0$.
Then the map $x^* = y$, $y^* = x$ and $0^* = 0$ can be shown to be an involution, and so $E_Y = Y \neq H_Y = \{0\}$.
Note that $Y$ does not form a regular ${}^{*}$\hyp{}semigroup when equipped with this involution. 

Note also that a commutative involution semigroup $S$ may have $S \neq S^2$.
For example, in $(\mathbb{N}, {+})$ with trivial involution we have $1 \notin \mathbb{N}^2 = \{2, 3, \dots\}$.\footnote{We follow the convention that $\mathbb{N}$ stands for the set of strictly positive integers.}

Since $S$ is commutative, every subset is reflexive, and hence it follows from Lemma~\ref{cor:morphic} that an involution subsemigroup $T$ of $S$ is equal to $\phi^{-1}(1)$ for some surjective  $({\circ},{}^{*})$\hyp{}homomorphism  
$\phi \colon S \twoheadrightarrow G$  onto a group $G$
if and only if $T$ is closed and $H_S \subseteq T$.

\begin{theorem} \label{thm:comm}
Let $S$ be a commutative involution semigroup and let $T \subseteq S$.
Then the following are equivalent: 
\begin{enumerate}[label={\upshape(\arabic*)}]
\item\label{thm:comm:1} $T$ is an HS\hyp{}stable involution subsemigroup of $S$.
\item\label{thm:comm:2} $H_S \subseteq T$, $T \omega \cap S^2 = T \cap S^2$ and $T \setminus S^2 = T^* \setminus S^2$.

Moreover, these conditions imply: 
\item\label{thm:comm:3} There exists a surjective $({\circ},{}^{*})$\hyp{}homomorphism $\phi \colon S \twoheadrightarrow G$ with $G$ a group such that $T\omega = \phi^{-1}(1)$.  
\end{enumerate}
\end{theorem}

\begin{proof}
\ref{thm:comm:1} $\boxed{\Rightarrow}$ \ref{thm:comm:2}.
By \ref{HS1} we have $H_S \subseteq T$ and the two other conditions follow from Theorem~\ref{thm: main}. 

\ref{thm:comm:2} $\boxed{\Rightarrow}$ \ref{thm:comm:1}.
Note that $x H_S^2 x^* = x x^* H_S^2$ for any $x \in S$, and so $\bigcup_{x \in S} x H_S^2 x^* = H_S^3$.
Moreover, $x x^* y y^* z z^* = (xyz) (z^* y^* x^*)$ by commutativity, and hence $H_S^3 \subseteq H_S$.
Therefore $x H_S^2 x^* \subseteq T$ for each $x \in S$, so $T$ is HS-stable by Theorem~\ref{thm: main}. 

\ref{thm:comm:2} $\boxed{\Rightarrow}$ \ref{thm:comm:3}.
Since $H_S^3 \subseteq H_S \subseteq T$, it follows from Corollary~\ref{cor:closedproperty2} that $T \omega$ forms a closed involution subsemigroup containing $H_S$.
The result then follows from Lemma~\ref{cor:morphic}. 
\end{proof}

We note that the implication \ref{thm:comm:3} $\Rightarrow$ \ref{thm:comm:2} in Theorem~\ref{thm:comm} needs not hold in general, as we will show at the end of Example~\ref{example comm}. 
Alternatively, commutative orthodox ${}^{*}$\hyp{}semigroups  $S$ which are not E-unitary provide further examples, since here $E_S$ is not HS-stable but $E_S\omega$ is closed and HS-stable.  

\begin{corollary}\label{thm:comm_gen}
Let $S$ be a commutative involution semigroup and let $A \subseteq S$.
Then $\genHS{A} = (\gen{A \cup H_S} \omega \cap S^2) \cup ((A \cup A^*) \setminus S^2)$.
\end{corollary}

\begin{proof}
As $H_S \subseteq \genHS{A}$ by \ref{HS1}, it follows from Theorem~\ref{thm: main_gen} that
\[
\genHS{A} = \genHS{A \cup H_S} = (((A \cup H_S) \cup H_S^3) \omega \cap S^2 ) \cup (((A \cup H_S) \cup (A \cup H_S)^*) \setminus S^2). 
\] 
Since $H_S^3 \subseteq H_S$ and $H_S = H_S^* \subseteq S^2$, the desired result follows.
\end{proof}

\begin{corollary}
Let $S$ be a commutative involution semigroup with $S = S^2$ and let $T \subseteq S$.
Then the following are equivalent: 
\begin{enumerate}[label={\upshape(\arabic*)}]
\item $T$ is an HS\hyp{}stable involution subsemigroup of $S$.
\item $H_S \subseteq T$ and $T$ is closed.
\item There exists a surjective $({\circ},{}^{*})$\hyp{}homomorphism $\phi \colon S \twoheadrightarrow G$ with $G$ a group such that $T = \phi^{-1}(1)$. 
\end{enumerate}
In particular, $\genHS{T} = \gen{T \cup H_S} \omega$.
\end{corollary} 

\begin{proof}
Follows immediately from Theorem~\ref{thm:comm} and Lemma~\ref{cor:morphic}. 
\end{proof}

\begin{example} \label{example comm}
Consider $S = (\mathbb{N}, {+})$ with trivial involution, noting that $H_S = 2 \mathbb{N}$ is a closed involution subsemigroup of $S$.
Hence $H_S$ is HS\hyp{}stable by Theorem~\ref{thm:comm}.
Now let $T$ be an HS\hyp{}stable involution subsemigroup of $\mathbb{N}$ containing $2k + 1$ for some $k \geq 0$.
If $k = 0$ then $2n + 1 \in (T \cup 2 \mathbb{N}) \omega \subseteq T$ for each $n \geq 0$ since $(2n + 1) + 1 = 2 (n + 1)$.
Hence $T = \mathbb{N}$.
Otherwise $k \geq 1$, so that $3 \in (T \cup 2 \mathbb{N}) \omega \subseteq T$ since $3 + (2k - 2) \in T$.
As $T$ is an involution subsemigroup containing $2 \mathbb{N} \cup \{3\}$, it follows that $T = \mathbb{N} + \mathbb{N} = \mathbb{N} + 1$.
We have thus shown that $\mathbb{N}$ has three HS\hyp{}stable involution subsemigroups: $2 \mathbb{N} \subsetneq \mathbb{N} + 1 \subsetneq \mathbb{N}$. 
 
Notice that $(\mathbb{N} + k)\omega = \mathbb{N}$ for any $k\in \mathbb{N}$ since $t+k+1,k+1\in \mathbb{N}+k$ for all $t\in \mathbb{N}$. 
Hence $\mathbb{N}+1$ provides an example of an HS\hyp{}stable but not closed involution subsemigroup. 
Moreover,  $\mathbb{N}+2$ is not HS\hyp{}stable but its closure is; this, together with the $({\circ},{}^{*})$\hyp{}homomorphism of $S$ onto the trivial group  provides a
counterexample to the implication \ref{thm:comm:3} $\Rightarrow$ \ref{thm:comm:1} of Theorem~\ref{thm:comm}.  
\end{example}


\section{Complex products in semilattices}
\label{sec:semilattice}

The characterisation from Section~\ref{grouplike} is useful if the involution semigroup has many HS\hyp{}stable involution subsemigroups.
We showed that, on the other end of the spectrum, which includes semilattices and monoids with zero (see Corollaries~\ref{cor:idempgen} and \ref{cor:monoid-zero}), no proper HS\hyp{}stable involution subsemigroups exist, making the statement of Proposition~\ref{characterisation} trivial.

We will now answer, by different means, Problem~\ref{generalproblem} for (meet) semilattices: given a semilattice $S$, for which subsets $S_1, \dots, S_n$ does the complex product $S_1 \cdots S_n$ and the union $S_1 \cup \dots \cup S_n$ generate the same subsemilattice of $S$. 

\begin{proposition}
\label{prop:semilattice}
Let $Y$ be a semilattice.
For any nonempty subsets $A_1, \dots, A_n$ of $Y$, the following are equivalent: 
\begin{enumerate}[label={\upshape(\arabic*)}]
\item\label{prop:semilattice:1} $\gen{A_1 \cup A_2 \cup \cdots \cup A_n} = \gen{A_1 A_2 \cdots A_n}$.
\item\label{prop:semilattice:2} For each $1 \leq i, j \leq n$ and every $\alpha_i \in A_i$, there exists $\beta_j \in A_j$ with $\alpha_i \leq \beta_j$.
\end{enumerate}
\end{proposition} 

\begin{proof}
\ref{prop:semilattice:1} $\boxed{\Rightarrow}$ \ref{prop:semilattice:2}
Let $\alpha_i \in A_i$.
Then as $\alpha_i \in \gen{A_1 \cup A_2 \cup \cdots \cup A_n} = \gen{A_1 A_2 \cdots A_n}$, there exist a $k \in \mathbb{N}$ and elements $\alpha_{pq} \in A_q$ ($1 \leq p \leq k$, $1 \leq q \leq n$) with
\[
\alpha_i = (\alpha_{11} \cdots \alpha_{1n}) \cdots (\alpha_{k1} \cdots \alpha_{kn}).
\]
For each $1 \leq j \leq n$, it follows by the commutativity of $Y$ that $\alpha_i = \alpha_{1j} \gamma$ for a suitable $\gamma$.
Hence $\alpha_i \leq \alpha_{1j} \in A_j$.

\ref{prop:semilattice:2} $\boxed{\Rightarrow}$ \ref{prop:semilattice:1}
It suffices to show that for each $1 \leq i \leq n$ and each $s_i \in A_i$ we have $s_i \in \gen{A_1 A_2 \cdots A_n}$.
By our hypothesis, for each $j \in \{1, \dots, n\}$ there exists $s_j \in A_j$ such that $s_i = s_i s_j$.
Then 
\begin{align*}
s_i
&= s_i^n
= (s_i s_1) (s_i s_2) \cdots (s_i s_i) (s_i s_{i+1}) \cdots (s_i s_n) = s_1 s_2 \cdots s_{i-1} s_i^{n+1} s_{i+1} \cdots s_n \\
& = s_1 s_2 \cdots s_{i-1} s_i s_{i+1} \cdots s_n
\in \gen{A_1 A_2 \cdots A_n}. 
\qedhere
\end{align*}
\end{proof}

 \begin{remark} 
Proposition \ref{prop:semilattice} has a particularly pleasing form if $Y$ has the ascending chain condition, i.e., has no infinite ascending chain $\alpha_1<\alpha_2<\alpha_3<\cdots$ of elements. 
In this case condition \ref{prop:semilattice:2} is equivalent to: 
\begin{itemize}[leftmargin=0.66cm]
\item[$(2)'$] The sets  $A_i$  have the same maximal elements. 
\end{itemize}
\end{remark} 


 \section*{Acknowledgements}

This work was partially supported by the Funda\c{c}\~ao para a Ci\^encia e a Tecnologia (Portuguese Foundation for Science and Technology) through the project UIDB/00297/2020 (Centro de Matem\'atica e Aplica\c{c}\~oes) and the project PTDC/\discretionary{}{}{}MAT-PUR/31174/2017.
The first and third author have received funding from the German Science Foundation (DFG, project number 622397)
and from the European Research Council (Grant Agreement no.\ 681988, CSP-Infinity).
 

\end{document}